\documentclass[12pt]{amsart}

\usepackage{amssymb}
\usepackage{graphicx}
\usepackage{enumerate}
\usepackage{multirow}
\usepackage{amsmath,color}
\usepackage{hyperref}
\usepackage{url}
\usepackage[colorinlistoftodos]{todonotes}
\usepackage{fullpage}

\usepackage{float}

\DeclareMathOperator{\correspond}{\sideset{_A}{_E}{\mathop{\longleftrightarrow}}\ }

\newtheorem{theorem}{Theorem}[section]
\newtheorem{corollary}[theorem]{Corollary}
\newtheorem{lemma}[theorem]{Lemma}
\newtheorem{proposition}[theorem]{Proposition}

\numberwithin{equation}{section}
\theoremstyle{definition}
\newtheorem{remark}{Remark}[section]
\newtheorem{example}{Example}

\newcommand{\BM}{Bannai-Muzychuk criterion}

\newcommand{\I}{\mathcal{I}}
\newcommand{\J}{\mathcal{J}}

\usepackage{tikz}
\usetikzlibrary{calc}
\usetikzlibrary{positioning}
\usetikzlibrary{chains}
\usetikzlibrary{er}
\usetikzlibrary{graphs}
\usetikzlibrary{cd}
\usetikzlibrary{shapes}
\usetikzlibrary{arrows}
\usetikzlibrary{matrix}
\usetikzlibrary{intersections}

\title[Amorphic association schemes and fusions of triples]{Characterizations of amorphic association schemes in terms of fusing triples}
\author{Yanzhen Xiong}
\address{College of Science, National University of Defense Technology, Changsha, Hunan, PR China}
\email{xiongyanzhen@nudt.edu.cn}

\begin{document}

\begin{abstract}
Let $\mathcal{R}$ be an association scheme with nontrivial relations $A_1,\ldots,A_d$. 
We call $\mathcal{R}$ amorphic if every possible fusion of its nontrivial relations gives rise to a fusion scheme. 
We define the fusing-relations $3$-hypergraph of $\mathcal{R}$ to be the $3$-uniform hypergraph on the vertex set $\{A_1,\ldots,A_d\}$ such that $\{ A_i, A_j, A_k \}$ forms an edge if it fuses, i.e., fusing $A_i, A_j, A_k$ gives rise to a fusion scheme of $\mathcal{R}$. 
A $3$-uniform hypergraph is called a $3$-sunflower if, for the edges, the union is the set of vertices and the intersection consists of $2$ vertices.
In this paper, we prove that for $d\geq 5$, $\mathcal{R}$ is amorphic if its fusing-relations $3$-hypergraph contains two $3$-sunflowers.  
As a corollary, for $d\geq 5$, $\mathcal{R}$ is amorphic if and only if all triples of its nontrivial relations fuse. 
\end{abstract}

\maketitle

\section{Introduction}

Let $X$ be a finite set, and let $d$ be a positive integer. For $0\leq i\leq d$, let $A_i$ be an $X\times X$ $\{0,1\}$ matrix. We call the configuration $\mathcal{R} = \{ A_i : \ i=0,1,\ldots,d \}$ a \emph{symmetric $d$-class association scheme} if the following holds. 
\begin{enumerate}
    \item $A_0$ is the identity matrix $I$;
    \item $\sum_{i=0}^{d} A_i$ is the all-one matrix $J$;
    \item $A_i^{\top} = A_i$ for all $i=1,\ldots,d$; 
    \item there are \emph{intersection numbers} $p_{ij}^h$ such that $A_iA_j = \sum_{h=0}^d p_{ij}^h A_h$ for all $i,j = 0,1,\ldots,d$.
\end{enumerate}
Throughout, all association schemes are symmetric. 

We mention that there are equivalent definitions of association schemes in terms of relations or the relation graphs. When no confusions arise, we will simply use the term relation for the relation itself, or for its relation graph, or for its adjacency matrix. For instance, we could say that a relation $A_i$ is strongly regular if its relation graph is a strongly regular graph. 
The identity relation $A_0 = I$ is referred to be the trivial relation. Unless otherwise specified, when we refer to a relation, it is always a non-trivial relation. 

An association scheme spans the \emph{Bose-Mesner algebra}, which is closed under both the ordinary matrix multiplication and the Hamadard product $\circ$ (i.e., the entrywise multiplication). The Bose-Mesner algebra has another basis consisting of minimal \emph{idempotents} $\{E_j : \ j = 0,1,\ldots,d\}$, which satisfy the following. 
\begin{enumerate}
    \item $E_0 = \frac{1}{\left| X \right|}J$; 
    \item $\sum_{j=0}^d E_j = I$;
    \item $E_j^{\top} = E_j$ for all $j=0,1,\ldots,d$;
    \item there are \emph{Krein parameters} $q_{ij}^h$ such that $E_i\circ E_j = \frac{1}{\left| X \right|} \sum_{h=0}^{d} q_{ij}^h E_h$ for all $i,j=0,1,\ldots,d$. 
\end{enumerate}

Let $\mathcal{R}$ be an association scheme with relations $A_0,A_1,\ldots,A_d$ and idempotents $E_0,E_1,\ldots,E_d$. 
The \emph{first eigenmatrix} $P = (P_{ij})_{0\leq i,j\leq d}$ and the \emph{second eigenmatrix} $Q = (Q_{ij})_{0\leq i,j\leq d}$ are square matrices such that 
\begin{enumerate}
    \item $A_i = \sum_{j=0}^d P_{ji} E_j$ for all $i=0,1,\ldots,d$;
    \item $E_j = \frac{1}{\left| X \right|} \sum_{i=0}^{d} Q_{ij} A_i$ for all $j=0,1,\ldots,d$. 
\end{enumerate}

Here are some basic properties of the eigenmatrices. 
It holds $P_{j0}=1$ for all $j$ and $P_{0i} = p_{ii}^0 = k_i$, where $k_i$ is the valency of relation $A_i$ for all $i$. $\{ P_{ji} : \ 1\leq j\leq d\}$ are the restricted eigenvalues of $A_i$ for all $1\leq i\leq d$. Dually, we have $Q_{i0}=1$ for all $i$, $Q_{0j} = q_{jj}^0 = m_j$, where $m_j$ is the rank of idempotent $E_j$ for all $j$, and $\{ Q_{ij} : \ 1\leq i\leq d \}$ are the restricted dual eigenvalues of $E_j$ for all $1\leq j\leq d$. 

The \emph{principal parts} of $P$ and $Q$ are obtained by removing the first row and the first column, so they contain the restricted eigenvalues and restricted dual eigenvalues, respectively. 
When we say that an eigenmatrix or its principle part has a certain matrix form, we mean that it can be made into the matrix via permutations on the columns and permutations on the rows. And if we write an equality between an eigenmatrix and a matrix, unless otherwise specified, it implies that the relations and idempotents are already in the fixed order due to the matrix. 

For a subset $\I$ of indices, we write $A_\I = \sum_{i\in \I} A_i$ and $E_\I = \sum_{i\in \I} E_i$. Let $\mathcal{R} = \{A_i : \ i = 0,1,\ldots, d\}$ be an association scheme, with idempotents $\{E_j : \ j = 0,1,\ldots,d\}$. Let $\pi = \{ \pi(i) : \ i=0,1,\ldots,d' \}$ be a partition of $\{0,1,\ldots,d\}$ where $\pi(0) = \{0\}$, and define $\mathcal{R}_\pi$ be the configuration $\{ A_{\pi(i)} : \ i = 0,1\ldots,d' \}$. We call $\mathcal{R}_\pi$ a \emph{fusion scheme} of $\mathcal{R}$ if $\mathcal{R}_\pi$ is also an association scheme. 
For more information about association schemes, we refer the readers to \cite{BBIT21,BI84}. 

A graph is strongly regular if it is regular and it has exactly two restricted eigenvalues (except for one eigenvalue that is equal to the valency). For a strongly regular graph $G$, if there are integers $n$ and $t$ such that $G$ has $v=n^2$ vertices, $t(n-1)$ valency, and $-t$ and $n-t$ as its restricted eigenvalues, then we call $G$ a strongly regular graph of Latin square type if $n$ and $t$ are positive, and call $G$ a strongly regular graph of negative Latin square type if $n$ and $t$ are negative. 

In 1985, Gol'fand, Ivanov and Klin  \cite{GIK94} introduced the amorphic association schemes, which are association schemes where all possible fusion of relations will lead to a fusion scheme. 
For an amorphic $d$-class association scheme with $d\geq 3$, Ivanov showed that all relations in the association scheme are strongly regular of Latin square type or negative Latin square type \cite{GIK94}. 
In 1991, Ito, Munemasa and Yamada showed that the amorphic association schemes are formally self-dual, i.e., $P = Q$, and that if all relations in an association scheme are either all strongly regular of Latin square type, or all strongly regular of negative Latin square type, then the association scheme is amorphic \cite{IMY91}. 
In 2010, Van Dam and Muzychuk proved that an association scheme is amorphic if and only if the principle part of one of its eigenmatrices has the following form \cite{vDM10}, which we call it the canonical form in \cite{vDKX25}. 
\[
\begin{bmatrix}
         b_1 & a_2 & a_3 & \cdots & a_{d} \\
         a_1 & b_2 & a_3 & \cdots & a_{d} \\ 
         a_1 & a_2 & b_3 & \cdots & a_{d} \\ 
         \vdots & \vdots & \vdots & \ddots & \vdots \\
         a_1 & a_2 & a_3 & \cdots & b_{d} \\
    \end{bmatrix}.
\]

As a consequence of these known results, for a $d$-class association scheme $\mathcal{R}$ with $d\geq 3$, it is amorphic if and only if there exist integers $n,t_1,\ldots,t_d$, which are either all positive or all negative, such that its eigenmatrix $P$ has the following form. 
\[\begin{bmatrix} 
1 & t_1(n-1) & t_2(n-1) & t_3(n-1) & \cdots & t_d(n-1) \\
1 & n-t_1 & -t_2 & -t_3 & \cdots & -t_d \\
1 & -t_1 & n-t_2 & -t_3 & \cdots & -t_d \\
1 & -t_1 & -t_2 & n-t_3 & \cdots & -t_d \\
\vdots & \vdots & \vdots & \vdots & \ddots  & \vdots \\
1 & -t_1 & -t_2 & -t_3 & \cdots & n-t_d \\
\end{bmatrix}\]
Note that if all the relations in $\mathcal{R}$ are strongly regular of Latin square type, the integers $n,t_1,\ldots,t_d$ are all positive, and if all the relations are strongly regular of negative Latin square type, the integers $n,t_1,\ldots,t_d$ are all negative. 
For several relations in an association scheme, we say that they are \emph{strongly regular of the same type} if they are all strongly regular of Latin square type, or all strongly regular of negative Latin square type. 

It is conjectured by Ivanov that an association scheme must be amorphic if all its relations are strongly regular. However, the conjecture is proved to be wrong when the class $d$ is larger than $3$. In particular, Van Dam proved that, for a non-amorphic $4$-class association scheme in which all relations are strongly regular, its the first eigenmatrix $P$ must have the following form \cite{D3}. 
\begin{align}\label{eq:Ephemeral}
\begin{bmatrix}
         1 & k_1 & k_2 & k_2 & k_2 \\
         1 & b_1 & a_2 & a_2 & a_2 \\ 
         1 & a_1 & b_2 & b_2 & a_2 \\ 
         1 & a_1 & b_2 & a_2 & b_2 \\
         1 & a_1 & a_2 & b_2 & b_2 \\
    \end{bmatrix}.
\end{align}
In the same paper \cite{D3}, Van Dam gave a example of association scheme with $P$ being the form \eqref{eq:Ephemeral}. Later on, Ikuta and Munemasa construct two more such examples \cite{IM08,IM10}. 

For a tuple of relations in an association scheme, we say that it \emph{fuses}, or call it a \emph{fusing tuple} if fusing it gives rise to a fusion scheme.
Let $k\geq 2$ be an integer. A \emph{$k$-uniform hypergraph}, or a $k$-hypergraph for short, is a hypergraph whose edges are all size-$k$ sets. 
the \emph{fusing-relations $k$-hypergraph} of an association scheme is the hypergraph whose vertices are the relations such that $k$ vertices form an edge if the $k$-tuple of the corresponding relations fuses. 
The fusing-relations $2$-hypergraphs are called the \emph{fusing-relations graph} \cite{vDKX25}. 
In 2025, Van Dam, Koolen and Xiong gives three characterizations of amorphic association schemes \cite{vDKX25}. 

\begin{theorem}\cite[Theorem~3.2, Theorem~4.1, Theorem~6.3]{vDKX25}\label{thm:Mellifluous}
    Let $d\geq 3$ and let $\mathcal{R}$ be a class-$d$ association scheme. Then $\mathcal{R}$ is amorphic if one of the following happens. 
    \begin{enumerate}
        \item Every pair of its relations fuses;
        \item Its fusing-relations graph is connected but not a path;
        \item There are at most one relation that is neither strongly regular of Latin square type nor strongly regular of negative Latin square type.  
    \end{enumerate}
\end{theorem}

\begin{remark}
    In \cite{vDKX25}, the authors give examples whose fusing-relations graphs are path $P_d$ or the disconnected graph $K_{d-1}\sqcup K_1$ for general $d$. They also dualize the results in terms of fusing-idempotents graph and strongly regular idempotents of Latin square type. 
\end{remark}

To generalize Theorem~\ref{thm:Mellifluous}~(3), Van Dam, Koolen and Xiong study the non-amorphic association scheme with exactly two relations that are neither strongly regular of Latin square type nor strongly regular of negative Latin square type \cite{vDKX26}. 

\begin{theorem}\cite[Theorem~5.1]{vDKX26}
  For all $d>1$, there exists a $d$-class association scheme with exactly $d-2$ relations that are strongly regular of Latin square type.     
\end{theorem}

In this paper, we aim to generalize Theorem~\ref{thm:Mellifluous}~(1) and (2).

A \emph{$k$-sunflower} is a $k$-hypergraph such that the intersection of all its edges is a set of size $k-1$, and the union of all its edges is the vertex set. The intersection is called the \emph{core}. Two $k$-sunflowers on a common vertex set are called \emph{different} if their cores are not the same. 
We call a hypergraph $\mathcal{H}_1$ is a subhypergraph of $\mathcal{H}_2$ if every edge of $\mathcal{H}_1$ is an edge of $\mathcal{H}_2$. 

\begin{figure}[htbp]
    \centering
    \begin{tikzpicture}[scale=2.5,
    dot/.style={circle, fill=black, minimum size=2mm, inner sep=0pt}]

  \node[dot] (p5) at (0, 1)    {}; \node at (0,    0.85)  {\small 5};
  \node[dot] (p6) at (0.901, 0.434) {}; \node at (0.85,  0.3)  {\small 6};
  \node[dot] (p7) at (0.901, -0.434) {}; \node at (0.9, -0.58)  {\small 7};
  \node[dot] (p1) at (0.434, -0.901) {}; \node at (0.334, -0.9)  {\small 1};
  \node[dot] (p2) at (-0.434, -0.901) {}; \node at (-0.334, -0.9) {\small 2};
  \node[dot] (p3) at (-0.901, -0.434) {}; \node at (-0.9, -0.58) {\small 3};
  \node[dot] (p4) at (-0.901, 0.434) {}; \node at (-0.85, 0.3)  {\small 4};
  
  \draw[thick]
    (-0.57,-1) .. controls (-0.6, -1) and (-1.2, -0.5) .. (-1,-0.334) .. controls (-0.9, -0.2) and (0.8, -0.8) .. (0.65,-1) .. controls (0.4, -1.2) and (-0.4, -1.2) .. cycle;
    
  \draw[thick] 
    (-0.54,-1) .. controls (-0.6, -1) and (-1.15, 0.55) .. (-1,0.6) .. controls (-0.7, 0.65) and (0.8, -1) .. (0.6,-1) .. controls (0.5, -1.15) and (-0.5, -1.15) .. cycle;

  \draw[thick] 
    (0.6,-1) .. controls (0.7, -1) and (0.3, 1.1) .. (0,1.1) .. controls (-0.3, 1.1) and (-0.7, -1) .. (-0.6,-1) .. controls (-0.5, -1.1) and (0.5, -1.1) .. cycle;

  \draw[thick] 
    (0.54,-1) .. controls (0.6, -1) and (1.15, 0.55) .. (1,0.6) .. controls (0.7, 0.65) and (-0.8, -1) .. (-0.6,-1) .. controls (-0.5, -1.15) and (0.5, -1.15) .. cycle; 

  \draw[thick]
    (0.57,-1) .. controls (0.6, -1) and (1.2, -0.5) .. (1,-0.334) .. controls (0.9, -0.2) and (-0.8, -0.8) .. (-0.65,-1) .. controls (-0.4, -1.2) and (0.4, -1.2) .. cycle;

\end{tikzpicture}
    \caption{A $3$-sunflower with vertex set $\{1,2,3,4,5,6,7\}$, edge set $\{ 123,124,125,126,127 \}$, and core $\{1,2\}$.}
    \label{fig:sunflower}
\end{figure}
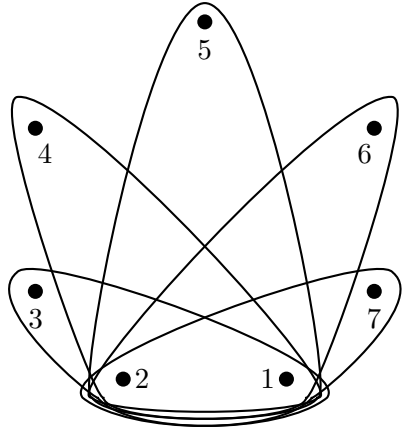

Here is the main result of this paper. We mention that the dual result is also true. 

\begin{theorem}\label{thm:six}
    Let $d\geq 5$. A $d$-class association scheme is amorphic if its fusing-relations $3$-hypergraph contains two different $3$-sunflowers as its subhypergraphs. 
\end{theorem}

\begin{corollary}\label{cor:five}
    Let $d\geq 5$. A $d$-class association scheme is amorphic if all triples of relations fuse. 
\end{corollary}

\begin{remark}
    For $4$-class association schemes, that its fusing-relation $3$-hypergraph is complete is equivalent to that all its relations are strongly regular. Therefore, there already exist non-amorphic $4$-class association schemes whose fusing-relation $3$-hypergraph is complete. 
\end{remark}

In the rest of the paper, we will prove our main results and their dual result, and conclude the paper with an open problem. 

\section{Generalizing the contraction lemma}

Let $\mathcal{R}$ be an association scheme with relations $A_0,A_1,\ldots,A_d$. Assume that $\{A_{i_1}, \ldots,A_{i_k}\}$ is a fusing tuple. 
We denote by $\mathcal{R}_{i_1,\ldots,i_k}$ the fusion scheme obtained from $\mathcal{R}$ by fusing the fusing tuple. 
We write $A_{i_1, \ldots, i_k}$ for the union of these relations, namely, $A_{i_1, \ldots, i_k} = A_{i_1} + \cdots + A_{i_k}$ as matrices. Hence, $A_{i_1, \ldots, i_k}$ is a relation in the association scheme $\mathcal{R}_{i_1,\ldots,i_k}$. Similarly, we define $E_{i,j} = E_i + E_j$ and so on. 

Van Dam et al. introduce the ``contraction lemma'' \cite[Lemma 5.1]{vDKX25}, which is a powerful tool in the study of fusing pairs of association schemes. 
Here is the key of the contraction lemma. 
Assume that $\{A_i, A_j\}$ and $\{ A_j, A_k\}$ are both fusing pairs of $\mathcal{R}$. Then $\{ A_{i,j}, A_k \}$ is a fusing pair of $\mathcal{R}_{i,j}$. 

In this section, we generalize the key of the contraction lemma for fusing triples, which will play a vital role in the proof of our main results.

\begin{lemma}\label{lem:contraction}
    Let $\mathcal{R}$ be a $d$-class association scheme with $d\geq 4$. Assume that $\{ A_i, A_j, A_k \}$ and $\{ A_j, A_k, A_\ell \}$ are two fusing triples of $\mathcal{R}$. Then $\{ A_{i,j,k}, A_\ell \}$ is a fusing pair of $\mathcal{R}_{i,j,k}$. 
\end{lemma}

To prove Lemma~\ref{lem:contraction}, we base on the \emph{\BM}, which states that, if $\mathcal{R}_\pi$ is a \emph{fusion scheme} of $\mathcal{R}$ for a given partition $\pi = \{ \pi(i) : \ i=0,1,\ldots,d' \}$ of $\{0,1,\ldots,d\}$ with $\pi(0) = \{0\}$, then there exists a unique partition $\rho = \{ \rho(i) : \ i=0,1,\ldots,d' \}$ of $\{0,1,\ldots,d\}$ with $\rho(0) = \{0\}$, such that each $(\rho(j), \pi(i))$-block of the first eigenmatrix $P$ has constant row sums. The idempotents of $\mathcal{R}_\pi$ are $\{ E_{\rho(j)} : \ j = 0,1,\ldots,d' \}$, and the $(\rho(j), \pi(i))$-entry of the first eigenmatrix of $\mathcal{R}_\pi$ equals the constant row sums of the $(\rho(j), \pi(i))$-block of the first eigenmatrix of $\mathcal{R}$. We mention that the \BM \ also works for the second eigenmatrix $Q$. 

For ease of notations, when we are using the \BM \ for fusion schemes, we always use $\pi$ for the partition of relations, and $\rho$ for the corresponding partition of idempotents. We continue to use the same notations, as in \cite{vDKX25}, 
\begin{align*}
    \I_1, \ldots, \I_p  \correspond  \J_1,\ldots,\J_q
\end{align*}
to show the correspondence, where $\I_1, \ldots, \I_p$ are all the partition blocks of $\pi$ with size $\geq 2$, and $\J_1,\ldots,\J_q$ are all the partition blocks of $\rho$ with size $\geq 2$. 

The following is a direct from \BM. 

\begin{proposition}\label{prop:Ubiquitous}
    Let $\mathcal{R}$ be a $d$-class association scheme with relations $A_0,A_1,\ldots,A_d$ and idempotents $E_0,E_1,\ldots,E_d$. Let $\I \subseteq \{1,\ldots,d\}$ be a set of size $3$. If $\{ A_i : \ i\in \I \}$ fuses, then exactly one of the two following happens. 
    \begin{enumerate}
        \item There is a set $\J \subseteq \{1,\ldots,d\}$ of size $3$ such that 
        \begin{align*}
           \I \correspond  \J.
        \end{align*}
        
        \item There are $\J_1, \J_2 \subseteq \{1,\ldots,d\}$ of size $2$ with $\J_1 \cap \J_2 = \emptyset$ and 
        \begin{align*}
           \I \correspond  \J_1, \J_2.
        \end{align*}
    \end{enumerate}
\end{proposition}

For a fusing triple $\{ A_i : \ i\in \I \}$, $\I \in \binom{\{1,\ldots,d\}}{3}$, we call that it is of \emph{type-$1$} if Proposition~\ref{prop:Ubiquitous}~(1) happens, and is of \emph{type-$2$} if Proposition~\ref{prop:Ubiquitous}~(2) does. 

\begin{example}\label{ex:fan}
    Let $\mathcal{R}$ be a $4$-class amorphic association scheme. Assume that the principle part of its second eigenmatrix $Q$ has the canonical form. 
Even though we have already know that $\mathcal{R}$ is amorphic, we may also use the \BM to check the fusions.
To see it more clearly, we draw a rectangles representing the submatrix $Q'$ indexed by the first three rows/relations and columns/idempotents. 
Observe that $Q'$ has constant row sum, and in the same rows of $Q'$, there are only constant columns outside $Q'$. According to the \BM, we have the fusion $ \{1,2,3\} \correspond \{1,2,3\} $. 
Similarly, we can also check $\{2,3,4\} \correspond \{2,3,4\}$. The two corresponding rectangles are depicted in Figure~\ref{fig:canonical}.
\end{example}

\begin{figure}[htbp]
    \centering
    \begin{tikzpicture}[scale=0.6,font=\small]
         \tikzstyle{vertex}=[draw,circle,minimum size=0pt,inner sep=0pt]
         \tikzstyle{edge} = [draw,thick,-]

    \draw (0,0) rectangle (3,-3); 
    \draw (1,-1) rectangle (4, -4);

    \node at (0.5,-0.5) {$b_1$};
    \node at (0.5,-1.5) {$a_1$};
    \node at (0.5,-2.5) {$a_1$};
    \node at (0.5,-3.5) {$a_1$};

    \node at (1.5,-0.5) {$a_2$};
    \node at (1.5,-1.5) {$b_2$};
    \node at (1.5,-2.5) {$a_2$};
    \node at (1.5,-3.5) {$a_2$};

    \node at (2.5,-0.5) {$a_3$};
    \node at (2.5,-1.5) {$a_3$};
    \node at (2.5,-2.5) {$b_3$};
    \node at (2.5,-3.5) {$a_3$};

    \node at (3.5,-0.5) {$a_4$};
    \node at (3.5,-1.5) {$a_4$};
    \node at (3.5,-2.5) {$a_4$};
    \node at (3.5,-3.5) {$b_4$};

    \draw[-] (-0.3,0) -- (-0.5,0) -- (-0.5,-4) -- (-0.3, -4); 
    \draw[-] (4.3,0) -- (4.5,0) -- (4.5,-4) -- (4.3, -4); 

    \end{tikzpicture}
\caption{The principle part of $Q$.}
\label{fig:canonical}
\end{figure}
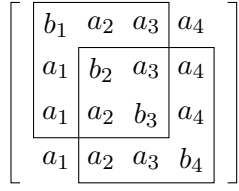

Now we look at Lemma~\ref{lem:contraction}. 
Without loss of generality, we assume $(i,j,k,\ell) = (1,2,3,4)$. We write $[d] = \{1,2,\ldots,d\}$. 
According to the types of the two fusing triples $\{A_1,A_2,A_3\}$ and $\{A_2, A_3, A_4\}$, 
we get Cases I, II and III, and for each of them, we have several possible subcases.  For each subcase, we choose a representative example. 

For each case, given $I,I_1, I_2, J, J_1, J_2$, we draw rectangles $\{ 1,2,3\} \times I, \{ 1,2,3\} \times I_1, \{ 1,2,3\} \times I_2$ and $\{ 2,3,4\} \times J, \{ 2,3,4\} \times J_1, \{ 2,3,4\} \times J_2$, so that all the cases are visualized in Figures~\ref{fig:dod1}~\ref{fig:dod2}~and~\ref{fig:dod3}. 
It will be helpful for the readers to check the proofs with the figures. 

{\bf Case I:} Both fusing triples are of type-$1$.

There are subsets $I, J \in \binom{[d]}{3}$ such that 
$\{ 1,2,3 \} \correspond I $ and $ \{2,3,4\} \correspond J$.
I will write $I = abc$ if $I = \{a,b,c\}$, and so on.  
There are $4$ subcases according to the size of $I\cap J$. 

Case I.1: $I = 123, J = 456$;  

Case I.2: $I = 123, J = 345$; 

Case I.3: $I = 123, J = 234$; 

Case I.4: $I = 123, J = 123$. 

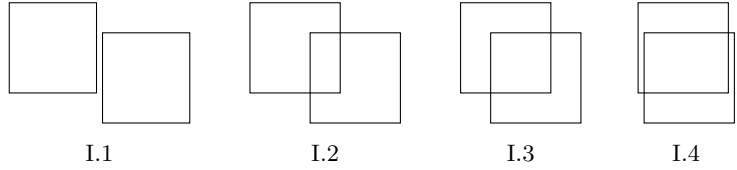
\begin{figure}[htbp]
    \centering
    \begin{tikzpicture}[scale=0.4,font=\scriptsize]
         \tikzstyle{vertex}=[draw,circle,minimum size=0pt,inner sep=0pt]
         \tikzstyle{edge} = [draw,thick,-]
    
    \def\margin{0.1}; 
    \def\gap{2};
    \def\start{-20};
    \def\starta{-20};
    \def\startb{-20};
    \def\startc{-20};
    
    \def\height{12};
    \def\heighta{6};
    \def\heightb{0};
    \def\heightc{-6};
    
         \draw (\start,\height) rectangle (\start+3-\margin,\height-3); 
         \draw (\start+3+\margin,\height-1) rectangle (\start+6, \height-4);

         \node at (\start+3,\height-5) {I.1};

         \draw (\start+6+\gap,\height) rectangle (\start+9+\gap,\height-3); 
         \draw (\start+8+\gap,\height-1) rectangle (\start+11+\gap, \height-4);

         \node at (\start+8.5+\gap,\height-5) {I.2};

         \draw (\start+11+2*\gap,\height) rectangle (\start+14+2*\gap,\height-3); 
         \draw (\start+12+2*\gap,\height-1) rectangle (\start+15+2*\gap, \height-4);

         \node at (\start+13+2*\gap,\height-5) {I.3};
         
         \draw (\start+15+3*\gap-\margin,\height) rectangle (\start+18+3*\gap-\margin,\height-3); 
         \draw (\start+15+3*\gap+\margin,\height-1) rectangle (\start+18+3*\gap+\margin, \height-4);

         \node at (\start+16.5+3*\gap,\height-5) {I.4};

    \end{tikzpicture}
    \caption{$4$ subcases of Case I.}
    \label{fig:dod1}
\end{figure}

{\bf Case II:} One of the fusing triple is of type-$1$ and the other is of type-$2$.  

There are subset $I \in \binom{[d]}{3}$ and $J_1,J_2 \in \binom{[d]}{2}$ with $J_1\cap J_2 = \emptyset$, such that 
$\{ 1,2,3 \} \correspond I $ and $ \{2,3,4\} \correspond J_1, J_2$.
There are $5$ subcases. 

Case II.1: $I = 123, J_1 = 45, J_2 = 67$.

Case II.2: $I = 123, J_1 = 34, J_2 = 56$.

Case II.3: $I = 123, J_1 = 23, J_2 = 45$.

Case II.4: $I = 234, J_1 = 12, J_2 = 45$.

Case II.5: $I = 123, J_1 = 12, J_2 = 34$.

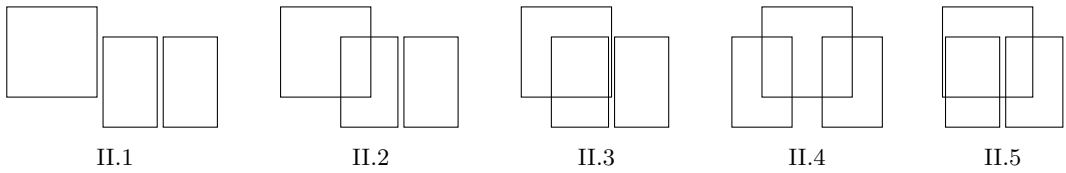
\begin{figure}[htbp]
    \centering
    \begin{tikzpicture}[scale=0.4,font=\scriptsize]
         \tikzstyle{vertex}=[draw,circle,minimum size=0pt,inner sep=0pt]
         \tikzstyle{edge} = [draw,thick,-]
    
    \def\margin{0.1}; 
    \def\gap{2};
    \def\start{-20};
    \def\starta{-20};
    \def\startb{-20};
    \def\startc{-20};
    
    \def\height{12};
    \def\heighta{6};
    \def\heightb{0};
    \def\heightc{-6};
         
         \draw (\starta-\margin,\heighta) rectangle (\starta+3-\margin, \heighta-3);
         \draw (\starta+3+\margin,\heighta-1) rectangle (\starta+5-\margin,\heighta-4);
         \draw (\starta+5+\margin,\heighta-1) rectangle (\starta+7-\margin,\heighta-4);

         \node at (\starta+3.5+0*\gap,\heighta-5) {II.1};

         \draw (\starta+7+\gap,\heighta) rectangle (\starta+10+\gap, \heighta-3);
         \draw (\starta+9+\gap,\heighta-1) rectangle (\starta+11 + \gap-\margin,\heighta-4);
         \draw (\starta+11 + \gap +\margin,\heighta-1) rectangle (\starta+13+\gap-\margin,\heighta-4);

         \node at (\starta+10+1*\gap,\heighta-5) {II.2};

         \draw (\starta+13+2*\gap,\heighta) rectangle (\starta+16+2*\gap, \heighta-3);
         \draw (\starta+14+2*\gap,\heighta-1) rectangle (\starta+16 + 2*\gap-\margin,\heighta-4);
         \draw (\starta+16 + 2*\gap +\margin,\heighta-1) rectangle (\starta+18+2*\gap-\margin,\heighta-4);

         \node at (\starta+15.5+2*\gap,\heighta-5) {II.3};

         \draw (\starta+19+3*\gap,\heighta) rectangle (\starta+22+3*\gap, \heighta-3);
         \draw (\starta+18+3*\gap,\heighta-1) rectangle (\starta+20 + 3*\gap,\heighta-4);
         \draw (\starta+21 + 3*\gap,\heighta-1) rectangle (\starta+23+3*\gap,\heighta-4);

         \node at (\starta+20.5+3*\gap,\heighta-5) {II.4};

         \draw (\starta+23+4*\gap,\heighta) rectangle (\starta+26+4*\gap, \heighta-3);
         \draw (\starta+23+4*\gap+\margin,\heighta-1) rectangle (\starta+25 + 4*\gap-\margin,\heighta-4);
         \draw (\starta+25 + 4*\gap+\margin,\heighta-1) rectangle (\starta+27+4*\gap,\heighta-4);

         \node at (\starta+25+4*\gap,\heighta-5) {II.5};

    \end{tikzpicture}
    \caption{$5$ subcases of Case II.}
    \label{fig:dod2}
\end{figure}

{\bf Case III:} Both fusing triples are of type-$2$. 

There are subsets $I_1, I_2, J_1, J_2 \in \binom{[d]}{2}$ with $I_1\cap I_2 = J_1 \cap J_2 = \emptyset$, such that 
$\{ 1,2,3 \} \correspond I_1, I_2 $ and $ \{2,3,4\} \correspond J_1, J_2$. 
There are $9$ subcases. 

Case III.1: $I_1 = 12, I_2 = 34, J_1 = 56, J_2 = 78$. 

Case III.2: $I_1 = 12, I_2 = 34, J_1 = 45, J_2 = 67$. 

Case III.3: $I_1 = 12, I_2 = 34, J_1 = 34, J_2 = 56$. 

Case III.4: $I_1 = 23, I_2 = 45, J_1 = 12, J_2 = 56$. 

Case III.5: $I_1 = 12, I_2 = 34, J_1 = 23, J_2 = 56$. 

Case III.6: $I_1 = 12, I_2 = 34, J_1 = 12, J_2 = 45$. 

Case III.7: $I_1 = 12, I_2 = 34, J_1 = 23, J_2 = 45$. 

Case III.8: $I_1 = 12, I_2 = 34, J_1 = 12, J_2 = 34$. 

Case III.9: $I_1 = 12, I_2 = 34, J_1 = 14, J_2 = 23$. 

\begin{figure}[htbp]
    \centering
    \begin{tikzpicture}[scale=0.4,font=\scriptsize]
         \tikzstyle{vertex}=[draw,circle,minimum size=0pt,inner sep=0pt]
         \tikzstyle{edge} = [draw,thick,-]
    
    \def\margin{0.1}; 
    \def\gap{2};
    \def\start{-20};
    \def\starta{-20};
    \def\startb{-20};
    \def\startc{-20};
    
    \def\height{12};
    \def\heighta{6};
    \def\heightb{0};
    \def\heightc{-6};

         \draw (\startb,\heightb) rectangle (\startb+2-\margin, \heightb-3);
         \draw (\startb+2+\margin,\heightb) rectangle (\startb+4-\margin, \heightb-3);
         \draw (\startb+4+\margin,\heightb-1) rectangle (\startb+6-\margin,\heightb-4);
         \draw (\startb+6+\margin,\heightb-1) rectangle (\startb+8-\margin,\heightb-4);

         \node at (\startb+4+0*\gap,\heightb-5) {III.1};

         \draw (\startb+8+\gap,\heightb) rectangle (\startb+10+\gap-\margin, \heightb-3);
         \draw (\startb+10+\gap+\margin,\heightb) rectangle (\startb+12 + \gap, \heightb-3);
         \draw (\startb+11 + \gap,\heightb-1) rectangle (\startb+13 + \gap-\margin,\heightb-4);
         \draw (\startb+13+\gap+\margin,\heightb-1) rectangle (\startb+15+\gap,\heightb-4); 

         \node at (\startb+11.5+1*\gap,\heightb-5) {III.2};

         \draw (\startb+15+2*\gap-\margin,\heightb) rectangle (\startb+17+2*\gap-\margin, \heightb-3);
         \draw (\startb+17+2*\gap,\heightb) rectangle (\startb+19 + 2*\gap -\margin, \heightb-3);
         \draw (\startb+17 + 2*\gap + \margin,\heightb-1) rectangle (\startb+19 + 2*\gap,\heightb-4);
         \draw (\startb+19+2*\gap + \margin,\heightb-1) rectangle (\startb+21+2*\gap,\heightb-4); 
         
         \node at (\startb+18+2*\gap,\heightb-5) {III.3};

         \draw (\startb+22+3*\gap,\heightb) rectangle (\startb+24+3*\gap-\margin, \heightb-3);
         \draw (\startb+24 +3*\gap+ \margin,\heightb) rectangle (\startb+26 + 3*\gap, \heightb-3);
         \draw (\startb+21+3*\gap,\heightb-1) rectangle (\startb+23+3*\gap,\heightb-4);
         \draw (\startb+25+3*\gap,\heightb-1) rectangle (\startb+27+3*\gap,\heightb-4);
         
         \node at (\startb+24+3*\gap,\heightb-5) {III.4};
         
         \draw (\startc+0+0*\gap+\margin,\heightc) rectangle (\startc+2+0*\gap-\margin, \heightc-3);
         \draw (\startc+2+0*\gap+\margin,\heightc) rectangle (\startc+4 -\margin + 0*\gap, \heightc-3);
         \draw (\startc+1 + 0*\gap,\heightc-1) rectangle (\startc+3 + 0*\gap -\margin,\heightc-4);
         \draw (\startc+3+0*\gap+\margin,\heightc-1) rectangle (\startc+5+0*\gap,\heightc-4); 

         \node at (\startc+3+0*\gap,\heightc-5) {III.5};

         \draw (\startc+5+\gap,\heightc) rectangle (\startc+7+\gap-\margin, \heightc-3);
         \draw (\startc+7+\gap+\margin,\heightc) rectangle (\startc+9+\gap, \heightc-3);
         \draw (\startc+5+\gap+\margin,\heightc-1) rectangle (\startc+7+\gap,\heightc-4);
         \draw (\startc+8+\gap+\margin,\heightc-1) rectangle (\startc+10+\gap,\heightc-4);

         \node at (\startc+7.5+1*\gap,\heightc-5) {III.6};

         \draw (\startc+10+2*\gap,\heightc) rectangle (\startc+12+2*\gap-\margin, \heightc-3);
         \draw (\startc+12+2*\gap+\margin,\heightc) rectangle (\startc+14+2*\gap, \heightc-3);
         \draw (\startc+11+2*\gap,\heightc-1) rectangle (\startc+13+2*\gap-\margin,\heightc-4);
         \draw (\startc+13+2*\gap,\heightc-1) rectangle (\startc+15+2*\gap,\heightc-4);

         \node at (\startc+12.5+2*\gap,\heightc-5) {III.7};

         \draw (\startc+15+3*\gap,\heightc) rectangle (\startc+17+3*\gap-\margin, \heightc-3);
         \draw (\startc+17+3*\gap+\margin,\heightc) rectangle (\startc+19+3*\gap, \heightc-3);
         \draw (\startc+15+3*\gap+\margin,\heightc-1) rectangle (\startc+17+3*\gap,\heightc-4);
         \draw (\startc+17+3*\gap+2*\margin,\heightc-1) rectangle (\startc+19+3*\gap+\margin,\heightc-4);

         \node at (\startc+18+3*\gap,\heightc-5) {III.8};

         \draw (\startc+20+4*\gap,\heightc) rectangle (\startc+22+4*\gap-\margin, \heightc-3);
         \draw (\startc+22+4*\gap+\margin,\heightc) rectangle (\startc+24+4*\gap, \heightc-3);
         \draw (\startc+21+4*\gap+\margin,\heightc-1) rectangle (\startc+23+4*\gap-\margin,\heightc-4);
         \draw[-] (\startc+20+4*\gap,\heightc-1) -- (\startc+21+4*\gap-\margin,\heightc-1) -- (\startc+21+4*\gap-\margin,\heightc-4) -- (\startc+20+4*\gap,\heightc-4);
         \draw[-] (\startc+24+4*\gap,\heightc-1) -- (\startc+23+4*\gap+\margin,\heightc-1) -- (\startc+23+4*\gap+\margin,\heightc-4) -- (\startc+24+4*\gap,\heightc-4);
         \draw[densely dotted] (\startc+24+4*\gap,\heightc-1) -- (\startc+25+4*\gap-\margin,\heightc-1) 
         (\startc+24+4*\gap,\heightc-4) -- (\startc+25+4*\gap-\margin,\heightc-4)
         (\startc+20+4*\gap,\heightc-1) -- (\startc+19+4*\gap+\margin,\heightc-1) 
         (\startc+20+4*\gap,\heightc-4) -- (\startc+19+4*\gap+\margin,\heightc-4);

         \node at (\startc+22+4*\gap,\heightc-5) {III.9};

    \end{tikzpicture}
    \caption{$9$ subcases of Case III.}
    \label{fig:dod3}
\end{figure}
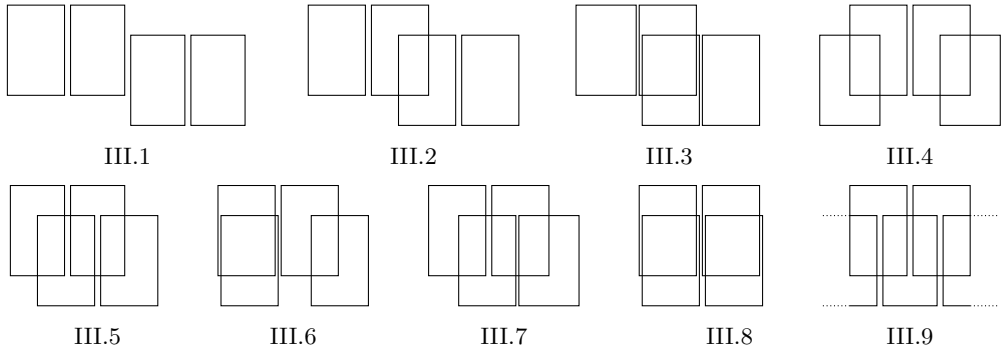

\begin{remark}
    Compare Figure~\ref{fig:canonical} with I.3 of Figure~\ref{fig:dod1}, and so the fusing triples $\{A_1,A_2,A_3\}$ and $\{A_2, A_3, A_4\}$ of the association scheme in Example~\ref{ex:fan} belongs to Case I.3. Note that the pictures in Figure~\ref{fig:canonical} reveals the submatrix of the second eigenmatrix $Q$ with constant row sums and without constant columns. If we consider the first eigenmatrix $P$, we need to exchange the rows and columns for each picture.  
\end{remark}

\begin{lemma}\cite[Lemma~2.1]{vDKX25}\label{lem:eigenmatrix}
    Let $M$ be the principal part of one of the eigenmatrices of an association scheme. Let $t\geq 2$. For any $t$ rows of $M$, there are at least $t$ columns of $M$ that are not constant on the $t$ rows. 
\end{lemma}

\begin{proof}[Proof of Lemma~\ref{lem:contraction}]

We look at the Case II.1. 
For the square $\{ 1,2,3 \} \times \{1,2,3\}$, it follows from the \BM \ that $Q_{1,j} = Q_{2,j} = Q_{3,j}$ for all $j$ with $1\leq j \leq d$ but $j\neq 1,2,3$, and it holds that $Q_{i,1}+Q_{i,2}+Q_{i,3}$ is constant for $i=1,2,3$. Similarly, for the squares $\{ 2,3,4 \} \times \{4,5\}$ and $\{2,3,4\} \times \{6,7\}$, we have that $Q_{2,j} = Q_{3,j} = Q_{4,j}$ for all $j$ with $1\leq j \leq d$ but $j\neq 4,5,6,7$, and it holds that $Q_{i,4}+Q_{i,5}$ is constant for $i=2,3,4$, and that $Q_{i,6}+Q_{i,7}$ is constant for $i=2,3,4$. In the proof, we will use the \BM \ for all the cases to get such equations over and over without listing out the equations. 

First we deal with the case I.4. 
For the four rows of the second eigenmatrix $Q$ indexed by $A_1, A_2, A_3, A_4$, it follows from \BM that there are at most $3$ non-constant columns. However, this contracts Lemma~\ref{lem:eigenmatrix}, which rules out the Case I.4. 

Next we turn to the Case III.3. 
By fusing $A_1, A_2, A_3$, let $Q'$ be the second eigenmatrix of the fusing scheme $\mathcal{R}_{1,2,3}$. 
According to the \BM, we have $Q'_{123,12} = Q_{3,1}+Q_{3,2} = Q_{4,1}+Q_{4,2} = Q'_{4,12}$, $Q'_{123,34} = Q_{3,3}+Q_{3,4} = Q_{4,3}+Q_{4,4} = Q'_{4,34}$, and for $j$ with $5\leq j \leq d$, $Q'_{123,j} = Q_{3,j} = Q_{4,j} = Q'_{4,j}$. Therefore, in $Q'$, the two rows indexed by $A_{123}$ and $A_{4}$ are exactly the same, which contradicts Lemma~\ref{lem:eigenmatrix}. The argument also works for the Case III.8. Hence, we prove that the Cases III.3 and III.8 are not possible. 

Next, for the Case I.1,
by the \BM, it is easy to see that $Q_{2,j} = Q_{3,j}$ for all $j$. Hence we find the two rows of $Q$ indexed by $A_2$ and $A_3$ are the same, contradicting Lemma~\ref{lem:eigenmatrix}. The same argument can also be applied to the Cases II.1 and III.1. This proves that the Cases I.1, II.1 and III.1 are not possible. 

For the Case I.2, by the \BM, we see that $Q_{2,j} = Q_{3,j}$ for all $1\leq j\leq d$ except for $j=3$. Recall that $PQ=QP=\left|X\right| I$ from \cite[Eq.~2.6]{BBIT21}. This implies that the principle part of $Q$ has constant row sum. This proves $Q_{2,3} = Q_{3,3}$, and so the two rows are the same. Similar argument can be applied to the Cases II.2 and III.2. Henceforth, the Cases I.2, II.2 and III.2 are impossible. 

For the Case II.4, by the \BM, we have equations $Q_{2,j} = Q_{3,j}$ for $j=1,3,5,6,\ldots,d$. Meanwhile, we have $Q_{2,1} + Q_{2,2} = Q_{3,1}+Q_{3,2}$ and $Q_{2,4} + Q_{2,5} = Q_{3,4} + Q_{3,5}$. These equations prove the two rows $A_2$ and $A_3$ are the same. Similar proof also work for the Cases III.4, III.5 and III.7. Consequently, the Cases II.4, III.4, III.5 and III.7 are also ruled out. 

It remains the Cases I.3, II.3, II. 5, III.6 and III.9. 
For convenience, we will always denote the second eigenmatrix of the fusion scheme to be $Q'$. 

For the Case I.3, if we fuse $A_1,A_2,A_3$, in the fusion scheme $\mathcal{R}_{1,2,3}$ it follows from the \BM \ that $Q'_{123,j} = Q_{3,j} = Q_{4,j} = Q'_{4,j}$ for all $j\geq 5$ and so in $\mathcal{R}_{1,2,3}$ we have $\{123,4\} \correspond \{123,4\}$. By symmetry, we know $\{A_1, A_{2,3,4}\}$ is a fusing pair of $\mathcal{R}_{2,3,4}$. 

For the Case II.3, in the fusion scheme $\mathcal{R}_{1,2,3}$, we have $Q'_{123,123} = Q_{3,1}+Q_{3,2}+Q_{3,3} = Q_{4,1}+Q_{4,2}+Q_{4,3} = Q'_{4,123}$, and $Q'_{123,j} = Q_{3,j} = Q_{4,j} = Q'_{4,j}$ for all $j\geq 6$. Therefore, we have $\{ 123,4 \}\correspond \{4,5\}$ in $\mathcal{R}_{1,2,3}$. In the fusion scheme $\mathcal{R}_{2,3,4}$, we have $Q'_{1,45} = Q_{1,4}+Q_{1,5} = Q_{3,4}+Q_{3,5} = Q'_{234,45}$, and $Q'_{1,j} = Q_{1,j} = Q_{3,j} = Q'_{234,j}$ for all $j\geq 6$. Hence, we have $\{1,234\}\correspond \{1,23\}$ in $\mathcal{R}_{2,3,4}$. 

For the Case II.5, in the fusion scheme $\mathcal{R}_{1,2,3}$, 
it holds $Q'_{123,j} = Q_{3,j} = Q_{4,j} = Q'_{4,j}$ for all $j\geq 5$. So we have $\{ 123,4 \}\correspond \{ 123,4 \}$ in $\mathcal{R}_{1,2,3}$. In the fusion scheme $\mathcal{R}_{2,3,4}$, it holds $Q'_{1,j} = Q_{1,j} = Q_{3,j} = Q'_{234,j}$ for all $j\geq 5$. Consequently, we have $\{1,234\}\correspond \{ 12,34 \}$ in $\mathcal{R}_{2,3,4}$. 

For the Case III.6, in the fusion scheme $\mathcal{R}_{1,2,3}$, we get $Q'_{123,12} = Q_{3,1}+Q_{3,2} = Q_{4,1}+Q_{4,2} = Q'_{4,12}$, and $Q'_{123,j} = Q_{3,j} = Q_{4,j} = Q'_{4,j}$ for all $j\geq 6$. Thus, we get $\{123,4\} \correspond \{ 34,5 \}$ in $\mathcal{R}_{1,2,3}$. By symmetry, we know $\{A_1, A_{2,3,4}\}$ is a fusing pair of $\mathcal{R}_{2,3,4}$. 

Finally, for the Case III.9, in the fusion scheme $\mathcal{R}_{1,2,3}$, it holds $Q'_{123,j} = Q_{3,j} = Q_{4,j} = Q'_{4,j}$ for all $j\geq 5$. Therefore, we have $\{123,4\}\correspond \{ 12,34 \}$. By symmetry, we know $\{A_1, A_{2,3,4}\}$ is a fusing pair of $\mathcal{R}_{2,3,4}$. 

Above all, for all the cases that are not ruled out (Cases I.3, II.3, II. 5, III.6 and III.9), we obtain that $\{A_{1,2,3}, A_4\}$ is a fusing pair of $\mathcal{R}_{1,2,3}$ and that $\{A_1, A_{2,3,4}\}$ is a fusing pair of $\mathcal{R}_{2,3,4}$, finishing the proof. 
\end{proof}

\section{Proof of the main theorem}

First we deal with the case of $d=5$. 

\begin{lemma}\label{lem:core}
    Let $\mathcal{R}$ be a $5$-class association scheme. Suppose its fusing-relations $3$-hypergraph contains a $3$-sunflower as its subhypergraph. Then the core of the $3$-sunflower is a fusing pair in $\mathcal{R}$. 
\end{lemma}

\begin{proof}
    We assume the edges of the $3$-sunflower are $\{ A_1, A_2, A_j \}, j=3,4,5$ so that $\{A_1, A_2\}$ is its core.  
    By Lemma~\ref{lem:contraction}, we know that $A_4$ and $A_{1,2,3}$ is a fusing pair in $\mathcal{R}_{1,2,3}$. Thus, $\mathcal{R}_{1,2,3,4}$ is a $2$-class fusion scheme with relations $A_{1,2,3,4}$ and $A_5$. This implies that $A_5$ is strongly regular. By symmetry, the relations $A_3, A_4$ and $A_5$ are strongly regular. In other words, for $i=3,4,5$, it holds $\left| \{ P_{1i}, \ldots, P_{5i} \} \right| = 2$,
    where $P$ is the first eigenmatrix of $\mathcal{R}$. 

    Suppose that we have a type-$1$ fusing triple. WLOG, we assume $\{1,2,3\}\correspond\{1,2,3\}$. By the \BM, we have $P_{1,j}=P_{2,j}=P_{3,j}$ for each $j=4,5$.  
    Since $A_3$ is strongly regular, by the Pigeonhole Principle, there exists $p,q \in \{ 1,2,3 \}$ with $p\neq q$ such that $P_{p,3} = P_{q,3}$. Therefore, it follows from the \BM \ that $\{1,2\} \correspond \{ p,q \}$, as desired. 

    Now we assume that all the edges of the $3$-sunflower are fusing triples of type-$2$. 
    WLOG, we first assume $\{ 1,2,3 \} \correspond \{ 1,2 \}, \{3,4\}$. By the \BM, we see that $P_{1,j} = P_{2,j}$ and $P_{3,j} = P_{4,j}$ hold for every $j= 4,5$. 
    If it holds $P_{1,3} = P_{2,3}$ or $P_{3,3} = P_{4,3}$, then by the \BM, we see that $\{1,2\} \correspond \{1,2\}$ or $\{ 1,2 \} \correspond \{3,4\}$, respectively. Then we are done. So in the rest of the proof, we assume that $P_{1,3}\neq P_{2,3}$ and $P_{3,3}\neq P_{4,3}$. 
    Recall that $A_3$ is strongly regular. This allows us to assume without losing any generality that $P_{1,3} = P_{3,3}$ and $P_{2,3} = P_{4,3}$. 

    Next we look at the fusing triple $\{ A_1,A_2, A_4 \}$. Because it is of type-$2$, we assume $\{ 1,2,4 \} \correspond I, J$. By the \BM, it must hold $P_{i,3} = P_{j,3}$ both for $i,j\in I$ and for $i,j\in J$. Since $d=5$, at least one of the two equations $I=\{1,3\}$ and $J = \{2,4\}$ must hold. For either equation, we obtain that $P_{1,5} = P_{2,5} = P_{3,5} = P_{4,5}$. 
    Similarly, applying the same argument to $\{A_1,A_2,A_5\}$, we get $P_{1,4} = P_{2,4} = P_{3,4} = P_{4,4}$. Thus, we get two constant columns for the first four rows of the principle part of $P$.  However, this is impossible because of Lemma~\ref{lem:eigenmatrix}. This is the proof.     
\end{proof}

\begin{theorem}\label{thm:Serendipity}
    Let $\mathcal{R}$ be a $5$-class association scheme. If its fusing-relations $3$-hypergraph contains two different $3$-sunflowers as its subhypergraphs, then $\mathcal{R}$ is amorphic. 
\end{theorem}

\begin{proof}
    If $A_1$ and $A_2$ form a fusing pair, then in the fusion scheme $\mathcal{R}_{1,2}$, $A_{1,2}$ and $A_i$ form a fusing pair for all $i=3,4,5$. By Theorem~\ref{thm:Mellifluous}~(2), we know that $\mathcal{R}_{1,2}$ is amorphic, and hence, $A_3, A_4, A_5$ are strongly regular of the same type. 
    Lemma~\ref{lem:core} allows us to find two fusing pairs of relations for $\mathcal{R}$, then there are at least $4$ of its relations that are strongly regular of the same type. It thus follows from Theorem~\ref{thm:Mellifluous}~(3) that $\mathcal{R}$ is amorphic. 
\end{proof}

\begin{proof}[Proof of Theorem~\ref{thm:six}]
    Thanks to Theorem~\ref{thm:Serendipity}, we only need to consider the case of $d\geq 6$. Denote the association scheme by $\mathcal{R}$. Pick a $3$-sunflower in the fusing-relations $3$-hypergraph of $\mathcal{R}$, and let $\{ A_i, A_j, A_k \}$ be an edge of the $3$-sunflower. Applying Lemma~\ref{lem:contraction}, we know that $A_{i,j,k}$ and $A_{\ell}$ form a fusing pair in the fusion scheme $\mathcal{R}_{i,j,k}$ for all $\ell \in \{1,\ldots,d\}\setminus\{i,j,k\}$. Since $d\geq 6$, from Theorem~\ref{thm:Mellifluous}~(2) we derive that $\mathcal{R}_{i,j,k}$ is amorphic, and so the relations $A_{\ell}, \ell\in \{1,\ldots,d\}\setminus\{i,j,k\}$ are strongly regular of the same type. 
    Apply the above argument for the other $3$-sunflower, we may find $\{i',j',k'\}$ such that $A_{\ell}, \ell\in \{1,\ldots,d\}\setminus\{i',j',k'\}$ are strongly regular of the same type. Because $d\geq 6$, we can always choose proper sets $\{i,j,k\}$ and $\{i',j',k'\}$ such that $\left| \{i,j,k\} \cap \{i',j',k'\} \right| \leq 1$. This implies that there are $d-1$ relations of $\mathcal{R}$ that are strongly regular of the same type. By Theorem~\ref{thm:Mellifluous}~(3), $\mathcal{R}$ must be amorphic. 
\end{proof}

\section{Dualizations}

Let $\mathcal{R}$ be an association scheme with relations $A_0,A_1,\ldots,A_d$ and idempotents $E_0,E_1,\ldots,E_d$. 
Let $Q$ be the second eigenmatrix of $\mathcal{R}$. Recall that for each $j = 1,\ldots,d$, the entries $Q_{i,j}, i = 1,\ldots,d$ are the restricted dual eigenvalues of $E_j$ \cite{Suda12}. 
An idempotent $E_j$ of $\mathcal{R}$ is strongly regular if it has exactly two restricted dual eigenvalues \cite{vDKX25}. 
We say that a tuple of idempotents $\{ E_j : \ j\in J \}$ \emph{fuses} if fusing it gives rise to a fusion scheme, i.e., there exists $I_1,\ldots,I_{\ell}$ such that $I_1,\ldots,I_{\ell} \correspond J$.   
As the dual of the fusing-relations $k$-hypergraph of $\mathcal{R}$, we define the \emph{fusing-idempotents $k$-hypergraph} of $\mathcal{R}$ to be the $k$-hypergraph with vertex set $\{E_1,\ldots,E_d\}$ such that $k$ vertices form an edge if they fuse.  

In \cite{vDKX25}, Van Dam, Koolen and Xiong have proved that the dual of Theorem~\ref{thm:Mellifluous} and the dual of Lemma~\ref{lem:eigenmatrix}. 
We also mention that we may apply the \BM \ to the second eigenmatrix $Q$ \cite{Muzychuk92}. 

Moreover, by interchanging relations and idempotents, and by interchanging $P$ and $Q$, we can easily prove the dual of Lemma~\ref{lem:contraction},  Lemma~\ref{lem:core} and Theorem~\ref{thm:Serendipity}, respectively. 

Finally, since all our tools are dualized, we get the following dual results. 

\begin{theorem}
    Let $d\geq 5$. A $d$-class association scheme is amorphic if its fusing-idempotents $3$-hypergraph contains two different $3$-sunflowers as its subhypergraphs. 
\end{theorem}

\begin{corollary}
    Let $d\geq 5$. A $d$-class association scheme is amorphic if all triples of idempotents fuse. 
\end{corollary}

\section{Further discussion}

In \cite{vDKX26}, Van Dam, Koolen and Xiong construct an infinite family of non-amorphic $d$-class association scheme with exactly $d-2$ relations that are strongly regular of Latin square type from the Brouwer-Pasechnik graphs. For such a $d$-class association scheme, its fusing-relations $3$-hypergraph is (isomorphic to) the hypergraph on the vertex set $\{1,\ldots,d\}$ and the edge set $\{ 12j : \ j=3,\ldots,d \} \cup \{ ijk : \ 3\leq i < j < k \leq d  \}$. In particular, it contains exactly one $3$-sunflower. Then what are the extremal graphs for the fusing-relations $3$-hypergraph to assure the association scheme to be amorphic?    

\section*{Acknowledgement}  
The author would like to thank Jack Koolen and Edwin van Dam for their comments that improved the presentation and quality of this article. 
This work is supported by the National Natural Science Foundation of China (No. 12501502) and Innovation Research Foundation of College of Science at National University of Defense Technology (202501-YJRC-LXY-01).

\bibliographystyle{plain}
\bibliography{ref}

\end{document}